\documentclass[reqno]{amsart}
\usepackage{amsmath,amsthm}
\usepackage{amssymb}

\numberwithin{equation}{section}

\newtheorem{Thm}[subsection]{Theorem}
\newtheorem{Lem}[subsection]{Lemma}
\newtheorem{Prop}[subsection]{Proposition}

\newtheorem{Rem}[subsection]{Remark}

\newcommand{\bc}{\mathbb{C}}
\newcommand{\bcp}{\mathbb C \mathbb P}
\newcommand{\brp}{\mathbb R \mathbb P}
\newcommand{\bz}{\mathbb Z}
\newcommand{\s}{\Sigma}
\newcommand{\wt}{\widetilde}

\begin{document}
\author{Ajay Singh Thakur}
\title[On trivialities of Stiefel-Whitney classes]{On trivialities of Stiefel-Whitney classes of vector bundles over iterated suspensions of Dold manifolds}


\address{School of Mathematics, Tata Institute of Fundamental Research,  Homi Bhabha Road, Colaba, Mumbai 400 005, India  
}
\email{athakur@math.tifr.res.in}

\thanks{Most of this work was done when the author was a post-doc at Indian Statistical Institute, Bangalore, India.}

\begin{abstract}
A space $X$ is called $W$-trivial if for every vector bundle $\xi$ over $X$, the total Stiefel-Whitney class $W(\xi)= 1$. In this article we shall investigate whether the suspensions of Dold manifolds, $\s^k D(m,n)$, is $W$-trivial or not.
\end{abstract}
\maketitle

\section{Introduction}
Recall \cite{tanaka} that a CW-complex $X$ is said to be  $W$-trivial if for any vector bundle $\xi$ over $X$, the total Stiefel-Whitney class $W(\xi)= 1$.

It is a theorem of Atiyah-Hirzebruch \cite[Theorem 2]{AH} that  the $9$-fold suspension $\s^9X$ of  any CW-complex $X$ is $W$-trivial (see also \cite[Corrollary  1.2]{tanaka_2010}). In the same paper, Atiyah-Hirzebruch \cite[Theorem 1]{AH} have shown that  the sphere $S^d =\s^d S^0$ is $W$-trivial if and only if $d \neq 1,2,4$ and 8 (see also \cite[Theorem 1]{milnor}). Here $S^0$ is the union to two distinct points.

It is therefore an interesting question to understand for what value of $k, 0\leq k \leq 8$, is the iterated suspension $\s^k X$, of a CW-complex $X$,  $W$-trivial. Another motivation to study the $W$-triviality of a CW-complex is its connection with $I$-triviality \cite{tanaka_2010}. If a CW-complex $B$ is $W$-trivial then it is $I$-trivial and hence it satisfies a Borsuk-Ulam type theorem. We refer to \cite{tanaka_2010} and \cite{tanaka_2003} for more details on $I$-triviality of a CW-complex.

In \cite{tanaka_2008}, R. Tanaka obtained results concerning the $W$-triviality and ``$W$-triviality except at one dimension" for highly connected CW-complexes. In \cite{tanaka_2010}, R. Tanaka  determined all pairs $(k,n)$ of positive integers for which $\s^k \mathbb F \mathbb P^n$ is $W$-trivial, where $\mathbb F = \mathbb R, \mathbb C$ or $\mathbb H$. 

In this article we shall investigate when the iterated suspension $\s^k D(m,n)$, of the Dold manifold $D(m,n)$ is $W$-trivial. Recall \cite{dold} that the Dold manifold $D(m,n)$ is an $(m+2n)$-dimensional manifold 
defined as the quotient of $S^m\times\bcp^n$ by the fixed point free involution $(x,z)\mapsto (-x,\bar{z})$. The projection $S^m\times \bcp^n\longrightarrow S^m$ gives rise to a fiber bundle 
\[\bcp^n\hookrightarrow D(m,n)\longrightarrow \brp^m\]
with fiber $\bc^n$ and structure group $\mathbb Z_2$. In particular, we have $D(m,0) = \brp^m$ and $D(0,n) = \bcp^n$.

By the  theorem of Atiyah-Hirzebruch \cite[Theorem 2]{AH}, $\s^kD(m,n)$ is $W$-trivial for $k \geq 9$. So we shall be interested only in the case $0 \leq k \leq 8$ and $m > 0$. We have the following main results.

\begin{Thm}\label{main1}
Let $\s^kD(m,n)$ be the $k$-fold suspension of the Dold manifold $D(m,n)$ with $m > 0$. Then $\s^kD(m,n)$ is not $W$-trivial if
\begin{enumerate}
\item $k = 0$.
\item $k = 1,2,4\mbox{ or } 8$ and $ m\geq k$.
\item $k = 3,5 \mbox{ or } 7$ and $m+k= 4 \mbox{ or }8$.
\item $k = 6$ \and $m = 2$ or $3$.
\end{enumerate}
\end{Thm}

\begin{Thm}\label{main2}
Let $\s^kD(m,n)$ be the $k$-fold suspension of the Dold manifold $D(m,n)$ with $m > 0$ and $n$ even. Then $\s^kD(m,n)$ is $W$-trivial if
\begin{enumerate}
\item $k = 2$ and $m = 1$.
\item $k = 3$ and $m \neq 5,8t+1$.
\item $k = 4$ and $m = 3$.
\item $k = 5$ and $m \neq  8t+3$.
\item $k = 6$ and $m \neq 2,3, 8t + 4$.
\item $k = 7$ and $m  \neq 1,8t + 5$.
\item $k = 8$ and $m = 1,2,3 \mbox{ or } 7$.
\end{enumerate}
\end{Thm}

We have the following result in the case when $n$ is odd.

\begin{Thm}\label{main3}
Let $\s^kD(m,n)$ be the $k$-fold suspension of the Dold manifold $D(m,n)$ with $m > 0$ and $n$ odd. Then $\s^kD(m,n)$ is $W$-trivial if $k$ and $m$ are as listed above in (1)-(7) of the Theorem \ref{main2} and any one of the following condition is satisfied: 
\begin{enumerate}
\item $n+k=  2,4 \mbox{ or }8$ and $m< k$.
\item $n+k=3,5 \mbox{ or } 7$ and $2n+m+k \neq 4,8$.
\item $n+k = 6$ and $m + n\neq 2,3$. 
\item $n+k \geq 9$.
\end{enumerate}
\end{Thm}

Observe that if we assume  $n \geq 3$ in the  Theorem \ref{main3}, then $\s^kD(m,n)$ is $W$-trivial except for  $\s^3D(m,5)$ with $m \geq 3$ and $\s^5D(m,3)$ with $m \geq 5$. In the case when $n = 1$ we have the following theorem.

\begin{Thm}\label{main4}
Let $\s^kD(m,1)$ be $k$-fold suspension of the Dold manifold $D(m,1)$ with $m> 0$. Then $\s^kD(m,1)$ is not $W$-trivial if
\begin{enumerate}
\item $k = 1,3 \mbox{ or }7$ and $m \geq k$.
\item $k = 2\mbox{ or }4$ and $m + k  = 2 \mbox{ or }6$.
\item $k = 5$ and $m = 1\mbox{ or }2$.
\end{enumerate}
\end{Thm}

To prove our results we shall require the description, by Fujii-Yasui \cite{fujiidold},  of $KO$-groups of Dold manifolds. We shall recall this in Section \ref{pre} and  prove our results in  Section \ref{result}.

%

\section{Preliminaries}\label{pre}
 In this section we shall recall the notations and results from \cite{fujiidold}, where  M. Fujii and T. Yasui have computed the $KO$-groups of Dold manifolds. These will be used to prove our results.

Let $\pi:D(m,n) \rightarrow D(m,n)/D(m,0)$ be the projection. Let $p: D(m,n) \rightarrow \brp^m$ be the projection map of the fiber bundle which is described in the introduction and let $i: D(m,0) \hookrightarrow D(m,n)$ be the inclusion defined by \[i([x_0,x_1,\ldots, x_m]) = [x_0,x_1,\ldots,x_m,1,0,\ldots,0].\] Consider the following exact sequence of $\wt{KO}$-groups,
\begin{equation} \label{longexact}
\cdots \rightarrow \wt{KO}^{-k}(D(m,n)/D(m,0)) \xrightarrow{\pi^!} \wt{KO}^{-k}(D(m,n)) \xrightarrow{i^!} \wt{KO}^{-k}(D(m,0)) \rightarrow \cdots
\end{equation}
Under the identification $D(m,0) = \brp^m$, we have the composition $i^! \circ p^! =$ identity. Hence the homomorphism 
\[p^!: \wt{KO}^{-k}(\brp^m) = \wt{KO}^{-k}(D(m,0)) \rightarrow \wt{KO}^{-k}(D(m,n))\] is an injective map and it gives the splitting of the exact sequence \eqref{longexact}. Let $\wt{KO}^{-k}(m,n) := \pi^!\wt{KO}^{-k}(D(m,n)/D(m,0))$. Then we have the following theorem.

\begin{Thm}\cite[Theorem 1]{fujiidold} \label{decomp}
$\wt{KO}^{-k}(D(m,n)) = \wt{KO}^{-k}(m,n) \oplus p^!\wt{KO}^{-k}(\brp^m),$ where $p: D(m,n) \rightarrow \brp^m$ is the natural projection.\qed
\end{Thm} 

The $KO$-groups of the projective space $\brp^m$ has been studied by M. Fujii in \cite{fujii_proj}. The group $\wt{KO}^{-k}(m,n)$ has been computed by M. Fujii and T. Yasui in \cite{fujiidold} by making use of the following two homeomorphisms \cite[Propsition 2]{fujii_complex}
\begin{enumerate}
\item $h_1: D(m,n)/D(m-1,n) \approx S^m \wedge (\bcp^n)^+$.
\item $h_2: D(m,n)/D(m,n-1) \approx S^n \wedge (\brp^{m+n}/\brp^{n-1})$.
\end{enumerate}
Here, for a space $X$, $X^+$ denotes the  disjoint union of $X$ and   a point. The identification of the spaces via homeomorphisms $h_1$ gives rise to the following long exact sequence  \cite[p. 58]{fujiidold},

\begin{eqnarray}\label{(m,n)}
\cdots \longrightarrow \wt{KO}^{-k}(S^m\wedge \bcp^n) \xrightarrow{f^!} \wt{KO}^{-k}(m,n) \xrightarrow{i^!} \wt{KO}^{-k}(m-1,n)   \nonumber \\ \xrightarrow{\delta} \wt{KO}^{-i+1}(S^m \wedge \bcp^n)\longrightarrow \cdots,
\end{eqnarray}
where $f = h_1 \circ \pi$ and $i: D(m-1,n) \hookrightarrow D(m,n)$ is the inclusion. The long exact sequence \eqref{(m,n)} is a direct summand of the following long exact sequence of $\wt{KO}$-groups for  the pair $(D(m,n),D(m-1,n))$,
\[\rightarrow \wt{KO}^{-k}(D(m,n)/D(m-1,n)) \xrightarrow{\pi^!} \wt{KO}^{-k}(D(m,n)) \xrightarrow{i^!} \wt{KO}^{-k}(D(m-1,n)) \rightarrow \cdots \]

In the case when $n = 2r$, the groups $\wt{KO}^{-k}(m,2r)$ have been described in Theorem 3 \cite{fujiidold}. The proof of the following lemma follows directly from the Theorem 3  of \cite{fujiidold} by counting the generators of $\wt{KO}^{-k}(m,2r)$.

\begin{Lem}
\label{zero} Let $\wt{KO}^{-k}(m,2r) := \pi^!\wt{KO}^{-k}(D(m,2r)/D(m,0))$, where $\pi$ is the projection. Let $m> 0$. Then  $\wt{KO}^{-k}(m,2r) = 0$ if,
\begin{enumerate}
\item $k =2$ and $m = 1$.
\item $k =3$ and $m = 8t+2, 8t+3, 8t+4 \mbox{ or } 8t+6$.
\item $k =5$ and $m = 8t, 8t+4, 8t+5 \mbox{ or } 8t+6$.
\item $k =6$ and $m = 8t + 1, 8t+ 5, 8t +6 \mbox{ or } 8t+7$.
\item $k =7$ and $m = 8t, 8t + 2, 8t + 6 \mbox{ or } 8t + 7$.
\end{enumerate} \qed
\end{Lem}

In the case when $n = 2r+1$, the long exact sequence of $\wt{KO}$-groups for the pair $(D(m,2r+1),D(m,2r))$ takes the following form \cite[p. 55]{fujiidold},
\begin{equation} \label{2r+1}
\rightarrow \wt{KO}^{-k}(D(m,2r+1)/D(m,2r)) \rightarrow \wt{KO}^{-k}(m,2r+1) \xrightarrow{i_1^!} \wt{KO}^{-k}(m,2r) \rightarrow,
\end{equation}
where $i_1: D(m,2r) \hookrightarrow D(m,2r+1)$ is the inclusion. For the long exact seqeunce \eqref{2r+1}, there exists  an algebraic splitting homomorphism \[\kappa: \wt{KO}^{-k}(m,2r) \rightarrow \wt{KO}^{-k}(m,2r+1)\] such that $i_1^! \circ \kappa =$ identity (refer Section 10 of \cite{fujiidold}). In fact, the homomorphism $\kappa$ is defined as the composition $i_2^! \circ p$, where $i_2:D(m,2r+1) \hookrightarrow D(m,2r+2)$ is the inclusion and $p:\wt{KO}^{-k}(m,2r) \rightarrow \wt{KO}^{-k}(m,2r+2)$ is an algebraic homomorphism defined in Section 10 of \cite{fujiidold}. Therefore we have the following theorem.

\begin{Thm}{\em (\cite[Theorem 2]{fujiidold})} \label{odd_decomp}
\[\wt{KO}^{-k}(m,2r+1) = \wt{KO}^{-k}(m,2r) \oplus \wt{KO}^{-k}(D(m,2r+1)/D(m,2r)).\] \qed
\end{Thm}

In the above direct sum decomposition, the group $\wt{KO}^{-k}(m,2r)$ is direct summand of $\wt{KO}^{-k}(m,2r+1)$ via the monomorphism $\kappa$, and the group $\wt{KO}^{-k}(D(m,2r+1)/D(m,2r))$ is isomorphic to $\wt{KO}^{-k}(S^{2r+1} \wedge (\brp^{2r+1+m}/\brp^{2r}))$ by the homeomorphism $h_2$. The later group, $\wt{KO}^{-k}(S^{2r+1} \wedge (\brp^{2r+1+m}/\brp^{2r}))$, has been computed by M. Fujii and T. Yasui in \cite{fujiistunted}.

\section{Proof of Main Results} \label{result}
 We first state the following well known facts which we shall use implicitly  in our proofs.
 \begin{enumerate}
 \item For a vector bundle $\xi$ over a CW-complex $X$, the smallest integer $k> 0$ with $w_k(\xi) \neq 0$ is a power of $2$ (see, for example, \cite{milnor}, p. 94 ).
 
 \item If $\wt{KO}(X)= 0$ then every vector bundle  over $X$ is stably trivial and hence $W(\xi) = 1$ for any vector bundle $\xi$ over $X$. Thus $X$ is $W$-trivial.
 
 \item Recall \cite{dold} that the  $\bz_2$-cohomology ring of the Dold manifold $D(m,n)$ is given  as
 \[ H^k(D(m,n);\bz_2) = \bz_2[c,d]/(c^{m+1} = 0, d^{n+1} = 0),\]
 where $c \in H^1(D(m,n);\bz_2)$ and $d\in H^2(D(m,n);\bz_2)$. If $c'$ is the generator of $H^1(\brp^m;\bz_2)$ and $d'$ is the generator of $H^2(\bcp^n;\bz_2)$ then $p^*(c') = c$ and $i^*(d) = d'$, where $p:D(m,n) \rightarrow \brp^m$ is the projection and $i:\bcp^n \hookrightarrow D(m,n)$ is the fibre inclusion of  the fibre bundle \[ \bcp^n \stackrel{i}\hookrightarrow D(m,n) \stackrel{p}\longrightarrow \brp^m.\] The action of Steenrod squares on the cohomology ring $H^*(D(m,n) ; \bz_2)$  are completely determined by the fact that  $Sq^1 d = cd$.
 
  \end{enumerate}

\begin{proof}[{ \em \textbf{Proof of Theorem \ref{main1}}}] Consider the projection map $p: D(m,n) \rightarrow \brp^m$. Since the composition map, \[\brp^m = D(m,0) \hookrightarrow D(m,n) \xrightarrow{p} \brp^m\] is the identity map, the induced map \[p^*: H^i(\brp^m;\bz_2) \rightarrow H^i(D(m,n);\bz_2)\] is injective. Hence the suspension map induces an injective map \[ \s^k p^*:H^i(\s^k \brp^m;\bz_2) \rightarrow H^i(\s^k D(m,n);\bz_2).\] Now if there is a vector bundle $\xi$ over $\s^k\brp^m$ with $w_i(\xi)\neq 0$, then $w_i(p^*(\xi))\neq 0$. Thus we have showed that if $\s^k\brp^m$ is not $W$-trivial then the $k$-fold suspension $\s^kD(m,n)$ of Dold manifold $D(m,n)$ is also not $W$-trivial. Thus the proof of the Theorem \ref{main1}, now follows from the Theorem 1.4 \cite{tanaka_2010}. \end{proof}

We now come to the proof of the Theorem \ref{main2}. The Theorem \ref{main2} will be proved in sequence  of propositions below. 
\begin{Prop} \label{prop1}
Let $m > 0$. Then $\s^kD(m,2r)$ is $W$-trivial if
\begin{enumerate}
\item $k = 2$ and $m = 1$.
\item $k = 3$ and $m = 8t+2, 8t+3, 8t+4 \mbox{ or } 8t+6$.
\item $k = 5$ and $m = 8t, 8t+4, 8t+5 \mbox{ or } 8t+6$.
\item $k = 6$ and $m = 8t + 1, 8t+ 5, 8t +6 \mbox{ or } 8t+7$.
\item $k = 7$ and $m = 8t, 8t + 2, 8t + 6 \mbox{ or } 8t + 7$.
\end{enumerate}
\end{Prop}
\begin{proof}Let $k$ and $m$ be as in the statement of the proposition.
Note that the group $\wt{KO}^{-k}(m,2r) = 0$ (Lemma \ref{zero}) and the $k$-fold suspension $\s^k\brp^m$ of $\brp^m$ is  $W$-trivial (Theorem 1.4 \cite{tanaka_2010}).  Now by the decomposition,
\[\wt{KO}^{-k}(D(m,2r)) = \wt{KO}^{-k}(m,2r) \oplus p^!\wt{KO}^{-k}(\brp^m),\] of the  Theorem \ref{decomp}, any vector bundle $\xi \in \wt{KO}^{-k}(D(m,2r))$ is stably  equivalent to $\eta \oplus \nu$, where $\eta \in \wt{KO}^{-k}(m,2r)$ and $\nu \in p^!\wt{KO}^{-k}(\brp^m)$. Since $W(\eta) = 1$ and $W(\nu) = 1$, we have $W(\xi) = 1$. This completes the proof of the proposition.
\end{proof}

\begin{Prop}\label{prop2}
Let $m > 0$. Then $\s^kD(m,2r)$ is $W$-trivial if
\begin{enumerate}
\item $k = 3$ and $m = 8t+5 (t>0) \mbox{ or } 8t+7$.
\item $k = 5$ and $m = 8t+1 \mbox{ or } 8t+7$.
\item $k = 6$ and $m = 8t \mbox{ or } 8t+2 (t>0)$ .
\item $k = 7$ and $m = 8t +1 \mbox{ or } 8t+3$.
\end{enumerate}
\end{Prop}
\begin{proof}Let $k$ and $m$ be as in the statement of the proposition.
Then consider the following exact sequence, \eqref{(m,n)}, 
\[ \cdots \rightarrow\wt{KO}^{-k}(S^{m} \wedge \bcp^{2r}) \xrightarrow{f^!} \wt{KO}^{-k}(m,2r) \xrightarrow{i^!} \wt{KO}^{-k}(m-1,2r) \rightarrow \cdots.\]
Since $\wt{KO}^{-k}(m-1,2r) = 0$ (Lemma \ref{zero}), the map $f^!$ is surjective. By Theorem 1.5 \cite{tanaka_2010}, $\s^{k+m}\bcp^{2r}$ is $W$-trivial and this implies that for any vector bundle $\xi \in \wt{KO}^{-k}(m,2r)$ the total Stiefel-Whitney class $W(\xi) = 1$. Furthermore, we know that $\s^{k}\brp^{m}$ is $W$-trivial (Theorem 1.4  \cite{tanaka_2010}). Hence by the decomposition \[\wt{KO}^{-k}(D(m,2r)) = \wt{KO}^{-k}(m,2r) \oplus p^!\wt{KO}^{-k}(\brp^m),\] of the  Theorem \ref{decomp}, we  conclude that $\s^kD(m,2r)$ is $W$-trivial.
\end{proof}

\begin{Prop}\label{prop3}
Let $m > 0$. Then $\s^kD(m,2r)$ is $W$-trivial if
\begin{enumerate}
\item $k = 3$ and $m = 8t$.
\item $k = 5$ and $m = 8t+2$.
\item $k = 6$ and $m = 8t+3$.
\item $k = 7$ and $m = 8t+4$.
\item $k = 8$ and $m = 1$.
\end{enumerate}
\end{Prop}
\begin{proof}Let $k$ and $m$ be as in the statement of the proposition.
Consider the inclusion, \[i: D(m-1,2r) \hookrightarrow D(m,2r)\]
Note that  the quotient $D(m,2r)/D(m-1,2r) \approx \s^{m}\bcp^{2r} \vee S^{m}$, and hence  the induced map \[i^*: H^p(D(m,2r);\bz_2) \rightarrow H^p( D(m-1,2r);\bz_2)\] is monomorphism when $p+m$ is odd.  Thus the induced suspension map \[{\s^qi}^*: H^q(\s^kD(m,2r);\bz_2) \rightarrow H^q(\s^kD(m-1,2r);\bz_2)\] is monomorphism for even $q$.

Now, if there is a vector bundle $\xi$ over $\s^kD(m,2r)$ with $w_{2^s}(\xi) \neq 0$ for some $s >0$, then $w_{2^s}(i^*(\xi)) \neq 0$. But we know that $\s^kD(m-1,2r)$ is $W$-trivial (refer Proposition \ref{prop2} for $k\neq 8$  and  Theorem 1.5 \cite{tanaka_2010} for $k=8$). This gives a contradiction and thus we conclude that $\s^kD(m,2r)$ is $W$-trivial.
\end{proof}

%
%
%
%
%

\begin{Prop}\label{prop4}
The iterated  suspensions $\s^8D(2,2r)$ and $\s^8D(3,2r)$ are $W$-trivial.
\end{Prop}
\begin{proof}
We shall first prove that $\s^8D(2,2r)$ is $W$-trivial. Let $\xi$ be a vector bundle over $\s^8D(2,2r)$. Let $s \geq 4$ be such that $2^s \leq \dim(\s^8D(2,2r))$ and  $w_{j}(\xi) = 0$ for $0 < j < 2^s $. We shall show that $w_{2^s}(\xi) = 0$ and thus this will imply that  $\s^8D(2,2r)$ is  $W$-trivial. 

Let $a \in H^{2^s -8}(D(2,2r);\bz_2)$ be the cohomology class which maps to $w_{2^s}(\xi) \in H^{2^s}(\s^8D(2,2r);\bz_2)$ under the suspension isomorphism
\[ H^{2^s -8}(D(2,2r);\bz_2) \rightarrow  H^{2^s}(\s^8D(2,2r);\bz_2).\] 
By Lemma 3.3 of \cite{tanaka_2010} and the fact that the Steenrod squares commutes with suspension homomorphism we have
\[Sq^i(a) = 0 \mbox{ for all }  0 < i < 2^{s-1}.\]
Now observe that the vector space $ H^{2^s -8}(D(2,2r);\bz_2)$ is generated by \[d^{2^{s-1}- 4} \mbox{ and } c^2d^{2^{s-1}-5}.\] Therefore, \[a = x \cdot d^{2^{s-1}} + y \cdot c^2d^{2^{s-1}}\] for $x,y \in \bz_2$. If $x \neq 0$ then $w_{2^s}(i^*\xi) \neq 0$, where \[ i : \s^8\bcp^{2r} \hookrightarrow \s^8 D(2,2r)\] is the inclusion map and $i^*\xi$ is the pullback bundle over $\s^8\bcp^{2r}$. But since $\s^8\bcp^{2r}$ is $W$-trivial (Theorem 1.5 \cite{tanaka_2010}), we have a contradiction and hence $x = 0$. Further, since \[Sq^2(c^2d^{2^{s-1}}) = c^2d^{2^{s-1}-4} \neq 0,\]
 we have $y = 0$. Hence $w_{2^s}(\xi) = 0$. This completes the proof of $W$-triviality of $\s^8D(2,2r)$.

The proof of $W$-triviality of $\s^8D(3,2r)$ proceeds along the same lines as the proof of Proposition \ref{prop3}  using the fact that $\s^8D(2,2r)$ is $W$-trivial.
\end{proof}

In the following proposition, $n$ can be both even or odd.
\begin{Prop}\label{prop5}
The iterated suspensions $\s^4D(3,n)$ and $\s^8D(7,n)$  are $W$-trivial.
\end{Prop} 
\begin{proof}Let $s\geq 3$ be such that $2^s \leq \dim(\s^4D(3,n))$.  Observe that  the vector space $H^{2^s - 4} (P(3,n);\bz_2)$ is generated by \[d^{2^{s-1} -2} \mbox { and } c^2d^{2^{s-1} -3}.\] Here note that $d^{2^{s-1} -2}$ will be zero if $n > 2^s -4$. Further observe that \[Sq^1(c^2d^{2^{s-1} -3}) = c^3d^{2^{s-1} -3 }\neq 0.\] With these observations the proof of the $W$-triviality of $\s^4D(3,n)$ proceeds along the same lines as the proof of $W$-triviality of $\s^8D(2, 2r)$ in the Proposition \ref{prop4}.

Similarly we can argue that $\s^8D(7,n)$ is $W$-trivial. Here we need to observe that, for $s \geq 4$ such that $2^s \leq \dim(\s^8D(7,n))$, the vector space $H^{2^s - 8} (P(7,n);\bz_2)$ is generated by 
\[d^{2^{s-1} -4},c^2d^{2^{s-1} - 5}, c^4d^{2^{s-1}-6} \mbox{ and } c^6d^{2^{s-1}-7}.\] Here again some of these cohomology classes, except $c^6d^{2^{s-1}-7}$, can be zero. Further observe that  $Sq^1(c^2d^{2^{s-1} - 5}) = c^3d^{2^{s-1} - 5}$, $Sq^1(c^6d^{2^{s-1} - 7}) = c^7d^{2^{s-1} - 7}$ and $Sq^1(c^4d^{2^{s-1} - 6}) = 0$ but $Sq^2(c^4d^{2^{s-1} -6})=c^6d^{2^{s-1} -6}$. Now the proof of $W$-triviality of $\s^8D(7,n)$ will proceed along the same lines as in the case $\s^4D(3,n)$. This completes the proof of the proposition.
\end{proof}

\begin{Rem}{\em
More generally one can prove that the $m$-fold suspension $\s^mD(m-1,n)$ of the Dold manifold $D(m-1,n)$ with $m>0$ is $W$-trivial by the method used in  proving the Proposition \ref{prop5}}.
\end{Rem}
The proof of the Theorem \ref{main2} follows from Propositions \ref{prop1}, \ref{prop2}, \ref{prop3}, \ref{prop4} and \ref{prop5}. 

We now come to the proof of the  Theorem \ref{main3}. First we make the following observation concerning the $W$-triviality of stunted projective space.


%
%
%
%
%

%

\begin{Lem} \label{stunted2}
Let $\brp^m/\brp^n$ be the stunted projective space with $m\geq n$. Then $\s^k(\brp^m/\brp^n)$ is $W$-trivial if 
\begin{enumerate}
\item $k= 1, 2,4 \mbox{ or }8$ and $m < k$.
\item $k=3,5 \mbox{ or } 7$ and $m+k \neq 4,8$.
\item $k = 6$ and $m \neq 2,3$. 
\item $k \geq 9$.
\end{enumerate}

\end{Lem}
\begin{proof}
Let $X=\brp^m/\brp^n$. Let $\alpha:\brp^m \rightarrow X$ be the projection map. Then the induced suspension homomorphsim \[(\s^k \alpha)^*:H^i(\s^kX;\bz_2) \rightarrow H^i(\s^k\brp^m;\bz_2),\] is an isomorphism  for $i >n+k$. Hence if there is a vector bundle $\xi$ over $\s^kX$ with $w_i(\xi) \neq 0$ then $w_i((\s^k\alpha)^*\xi) \neq 0$. Thus the $W$-triviality of $\s^k\brp^m$ implies the $W$-triviality of $\s^kX$. Now the proof of the lemma follows from the Theorem 1.4 \cite{tanaka_2010}.
\end{proof}

\begin{Prop}\label{odd}
Let $\s^kD(m,2r+1)$ be $k$-fold suspension of the Dold manifold $D(m,2r+1)$. Then  $\s^kD(m,2r+1)$ is $W$-trivial if $\s^kD(m,2r+2)$ and $\s^{2r+1+k} (\brp^{2r+1+m}/\brp^{2r})$ are $W$-trivial.
\end{Prop}
\begin{proof}
Let $\xi$ be a vector bundle over $\s^kD(m,2r+1)$. We shall prove the proposition by showing that the total Stiefel-Whitney class $W(\xi) = 1$.

Consider the following decomposition by Theorem \ref{decomp},
\[\wt{KO}^{-k}(D(m,2r+1)) = \wt{KO}^{-k}(m,2r+1) \oplus \wt{KO}^{-k}(\brp^m).\]
Since the $W$-triviality of $\s^kD(m,2r+2)$ implies the $W$-triviality of  $\s^k\brp^m$ (Theorem \ref{decomp}), we can assume that $\xi \in \wt{KO}^{-k}(m,2r+1)$. Now, since the $(2r+1+k)$-fold suspension $\s^{2r+1+k} (\brp^{2r+1+m}/\brp^{2r})$ is $W$-trivial and by the decomposition,  
       \[\wt{KO}^{-k}(m,2r+1) = \wt{KO}^{-k}(m,2r) \oplus \wt{KO}^{-k}(\s^{2r+1}(\brp^{2r+1+m}/\brp^{2r})),\]
of the Theorem \ref{odd_decomp}, we can further assume that $\xi = \kappa(\gamma)$ for some $\gamma \in \wt{KO}^{-k}(m,2r)$. Here $\kappa$ is the monomorphism with respect to which $\wt{KO}^{-k}(m,2r)$ is direct summnad of $\wt{KO}^{-k}(m,2r+1)$ (refer to section 10 of \cite{fujiidold} for more details). By the definition of  $\kappa$, we have $\xi = i_2^!(\eta)$ for some $\eta \in \wt{KO}^{-k}(m,2r+2)$. Here $i_2:D(m,2r+1) \hookrightarrow D(m,2r+2)$ is the inclusion. Thus, by the $W$-triviality of $\s^k D(m,2r+2)$, we conclude that $W(\xi) = 1$. This completes the proof of the proposition. 
\end{proof}

%

Now, under the given hypothesis on $k, n =2r+1$ and $m$, as in statement of  the Theorem \ref{main3}, the $k$-fold suspension $\s^kD(m,2r+2)$ is $W$-trivial for all $r\geq 0$ (Theorem \ref{main2}). It is also clear by the Lemma \ref{stunted2} that for these values of $k,n=2r+1$ and $m$, the space  $\s^{2r+1+k} (\brp^{2r+1+m}/\brp^{2r})$ is $W$-trivial. Thus the proof of the Theorem \ref{main3} follows from the Proposition \ref{odd}. 

We now come to the proof of  the Theorem \ref{main4}. First note the following lemma.
\begin{Lem} \label{stunted1}
The $k$-fold suspension $\s^kD(m,n)$ is not $W$-trivial if the $(n+k)$-suspension $\s^{n+k} (\brp^{n+m}/\brp^{n-1})$ is not $W$-trivial.
\end{Lem}
\begin{proof}
Observe that the inclusion $i: D(m,n-1) \hookrightarrow D(m,n)$ induces a surjective map 
       \[i^*: H^p(D(m,n);\bz_2) \rightarrow H^p(D(m,n-1);\bz_2), \mbox{ for all } p,\] 
in cohomology group.  Thus 
        \[\pi^*:H^p(D(m,n)/D(m,n-1);\bz_2) \rightarrow H^p(D(m,n);\bz_2), \mbox{ for  } p> 0,\] is injective. Here \[\pi: D(m,n) \rightarrow D(m,n)/D(m,n-1) \approx \s^{n} (\brp^{n+m}/\brp^{n-1})\] is the quotient map. Thus the induced suspension morphism
        \[ \s^k\pi^*: H^p(\s^{n+k}( \brp^{n+m}/\brp^{n-1});\bz_2) \rightarrow H^p(\s^kD(m,n);\bz_2)\] is injective for $p> 0$. Hence if there is a vector bundle $\xi$ over 
$\s^{n+k}(\brp^{n+m}/\brp^{n-1})$ with $W(\xi) \neq 1$, then $W(\pi^*(\xi)) \neq 1$. This completes the proof of the Lemma.
\end{proof}

Now the proof of the Theorem \ref{main4} follows immediately from  Lemma \ref{stunted1} and Theorem 1.4 \cite{tanaka_2010}.

This completes the proof of the results stated in the introduction.  As noted earlier we still do not know whether $\s^3D(m,5)$ and $\s^5D(m,3)$  are $W$-trivial or not for all $m$. Though this can be answered in few cases, we do not have a complete picture. For example if $m = 2,3 \mbox{ or } 4$ then $\s^3D(m,5)$ is $W$-trivial by  Proposition \ref{odd} and Corollary 1.3 of \cite{tanaka_2010}. By similar argument we can say that if $m = 1,2,4,5 \mbox{ or } 6$ then $\s^5D(m,3)$ is $W$-trivial. We also have the following proposition.

\begin{Prop}\label{main5}
Let $n>1$ and $n \not \equiv 3 \pmod 4$. Then $\s^4D(1,n)$ is not $W$-trivial.
\end{Prop}
\begin{proof}

Consider the following long exact sequence, \eqref{(m,n)},
 \[\cdots \rightarrow \wt{KO}^{-4}(\s\bcp^{n}) \xrightarrow{f^!} \wt{KO}^{-4}(1,n) \xrightarrow{i^!} \wt{KO}^{-4}(0,n) \xrightarrow{\delta} \wt{KO}^{-3}(\s\bcp^{n}) \rightarrow \cdots.\]
 Since   $\wt{KO}^{-4}(\s\bcp^n) = 0$ \cite[Theorem 2]{fujii_proj}, the homomorphism \[i^!: \wt{KO}^{-4}(1,n) \longrightarrow \wt{KO}^{-4}(0,n)\] is a monomorphism. We shall first prove that $i^!$ is an isomorphism.
  
Depending upon whether $n$ is odd or even, we write $n = 2r$ or $2r+1$.  As $n \not \equiv 3 \pmod 4$ we have, by  Theorem 2 \cite{fujii_proj}, 
\[\wt{KO}^{-4}(\bcp^{n}) = \wt{KO}^{-4}(0,n) = \bz^r.\]
Now if $n = 2r+ 1$ then we have the decompostion,
 \[ \wt{KO}^{-4}(1,2r+1) = \wt{KO}^{-4}(1,2r) \oplus \wt{KO}^{-4}(\s^{2r+1}(\brp^{2r+2}/\brp^{2r})),\] by the Theorem \ref{odd_decomp}. We have $\wt{KO}^{-4}(\s^{2r+1}(\brp^{2r+2}/\brp^{2r})) = 0$ (refer Table (2) on p. 47 of \cite{fujiistunted}) and $\wt{KO}^{-4}(1,2r) = \bz^r$ \cite[Theorem 3]{fujiidold}. Therefore,
 \[\wt{KO}^{-4}(1,n) =  \bz^r.\] 
 Now putting all these values of $\wt{KO}$-groups in the above long exact sequence one can easily conclude that the homomorphism \[\delta :\wt{KO}^{-4}(0,n) \longrightarrow \wt{KO}^{-3}(\s\bcp^{n}) = \wt{KO}^{-4}(\bcp^n)\] is zero. Thus the homomorphism $i^!$ is an isomorphism. 
 
 As $\s^4\bcp^{n}$ is not $W$-trivial (Theorem 1.5 \cite{tanaka_2010}), there is a vector bundle $\xi \in \wt{KO}^{-4}(0,n) = \wt{KO}^{-4}(\bcp^{n})$ with   $W(\xi) \neq 1$. Thus there is a vector bundle over $\s^4D(1,n)$ with non-trivial Stiefel-Whitney class.  This completes the proof of the proposition.

\end{proof}

\begin{Rem}{\em
Note that there is no integer $s$ such that $5 \leq 2^s \leq \dim( \s^4D(1,1)) = 7$. Hence $\s^4D(1,1)$ is $W$-trivial.}
\end{Rem}

Following are the cases which we have not been able to settle: (i) $k = 3$ and $m = 8t+1$
(ii) $k = 4$ and $m =2$ (iii) $k = 5$ and $m = 8t+3$
(iv) $k = 6$ and $m = 8t + 4$ (v) $k = 7$ and $m = 8t + 5$
(vi) $k = 8$ and $m = 4,5 \mbox{ or }6$. 
In addition to these cases,  we also do not know whether $\s^kD(m,n)$ is $W$-trivial or not when $k$, $m$ and $n$ satisfy any one of the following condition: (i) $k = 3$, $n = 5$ and $m \geq 5$ (ii) $k = 5$, $n = 3$  and $ m \geq 7$. (iii) $k = 4$, $ m = 1$ and $ n \equiv  3\pmod 4$.

\subsection*{Acknowledgement.}
The author is grateful to Aniruddha Naolekar for several valuable discussions.  He also thanks him for his helpful comments and suggestions on the manuscript.

\end{document}